\documentclass[graybox]{svmult}

\usepackage{type1cm}        
\usepackage{makeidx}         
\usepackage{graphicx}        
\usepackage{multicol}        
\usepackage[bottom]{footmisc}

\makeindex             

\usepackage{amssymb,amsmath}

\def\la{\langle}
\def\ra{\rangle}

\def\Gr{Gr\"obner}

\def\HS{{\mathrm{HS}}}

\def\LM{{\mathrm{LM}}}
\def\lm{{\mathrm{lm}}}

\def\Tor{{\mathrm{Tor}}}
\newcommand{\Z}{\ensuremath{\mathbb Z}}

\begin{document}

\title*{Context-free languages and associative algebras with algebraic
Hilbert series}

\titlerunning{Context-free languages and associative algebras}

\author{Roberto La Scala and Dmitri Piontkovski}
\authorrunning{R. La Scala and D. Piontkovski}

\institute{Roberto La Scala 
	\at Dipartimento di Matematica,
	Universit\`a di Bari, Via Orabona 4, 70125 Bari, Italy, 
    \email{roberto.lascala@uniba.it}
	\and 
	Dmitri Piontkovski 
	\at Department of Mathematics for Economics,
	National Research University Higher School of Economics,
	Myasnitskaya ul. 20, Moscow 101000, Russia 
	\email{dpiontkovski@hse.ru}
}

%
%

\maketitle

\abstract{In this paper, homological methods together with the theory of formal languages of theoretical computer science are proved to be effective tools to determine the growth and the Hilbert series of an associative algebra. Namely, we construct a class of finitely presented associative algebras related to a family of context-free languages. This allows us to connect the Hilbert series of these algebras with the generating functions of such languages. In particular, we obtain a class of finitely presented graded algebras with non-rational algebraic Hilbert series.}

\section{Introduction}
\label{sec:intro}

The Hilbert series (or growth series) of graded and filtered structures
is one of the most important invariants for infinite dimensional
algebraic objects. In particular for associative algebras, such series is
the most natural tool for finding the growth. 
For a number of important classes of algebras, Hilbert series are of special form, so that they are useful to characterize Koszul algebras, Noetherian algebras, some classes of PI algebras, algebras of small homological dimensions (such as noncommutative complete intersection) and many other algebras
(see, for instance \cite{KL, Uf2}). 

For many important classes, the Hilbert series is a rational function.
It was Hilbert himself who proved this property for
(finitely generated) commutative algebras.
Govorov \cite{Go1} proved in 1972 that
finitely presented monomial algebras have rational Hilbert series.
After that, the rationality of Hilbert series has been proved for a number of classes of associative algebras, such as, for example, 
prime PI algebras~\cite{Bell_prime_PI} and relatively free PI algebras~\cite{Belov_rat_PI}.

Moreover, Ufnarovski have introduced~\cite{Uf1} a general class of algebras with rational Hilbert series by connecting the theory of algebras with the theory of formal languages of theoretical computer science. The regular languages are the ones
that are recognized by finite-state automata and it is well-known that
such languages have rational generating functions (see, e.g., \cite{KS}).
A finitely generated algebra defined by a monomial set of relations is called automaton
(or Ufnarovski automaton) if the set of relations forms a regular language,
see details in Subsection~\ref{subs:automaton}. Moreover, the set of normal words
of such an algebra is also a regular language.  Since regular languages are known
to have rational generating functions, the Hilbert series of automaton algebras
are always rational. Moreover, the finitely presented monomial algebras are automaton,
so that the Govorov's rationality theorem follows from this result as a particular case.
Optimal algorithms to compute the rational (univariate and multivariate) Hilbert series
of an automaton algebra are due to La Scala and Tiwari \cite{LS2,LT}.

Govorov had conjectured in 1972 that all finitely presented graded algebras
have rational series. However, a couple of counterexamples were found
by Shearer \cite{Sh} in 1980 and Kobayashi \cite{Ko} in 1981. 
We remark that the corresponding non-rational Hilbert series were algebraic
functions, that is, roots of polynomials with coefficients in the rational
function field. At the same time, classes of finitely presented universal enveloping
algebras with intermediate growth (having hence transcendental Hilbert series)
were also discovered in~\cite{Uf2}. Examples of such algebras have been recently
introduced also by Ko\c{c}ak \cite{Kc1,Kc2}. Other examples of finitely presented
algebras with transcendental Hilbert series are recently given in~\cite{IS, P2019}.
Note that for important classes of algebras (such as Koszul algebras or graded
Noetherian algebras) the question about rationality of Hilbert series is still open.

Whereas a number of algebras with either rational or transcendental Hilbert series are known, there are only few examples of algebras with non-rational {\em algebraic} Hilbert series. We have therefore defined the class of {\em homologically unambiguous} algebras~\cite{our_jsc_paper}. These are the monomial algebras such that their relations together with the monomial bases of their homologies are unambiguous context-free languages (see details in Subsection~\ref{subs:unambiguous}). If such an algebra has finite homological dimension, then its Hilbert series is algebraic. An example of a finitely presented graded algebra with non-rational algebraic Hilbert series is constructed in~\cite{our_jsc_paper}.
 This is in fact an algebra such that the associated monomial algebra is homologically unambigous. 

In this paper, we give a general method to construct a finitely presented
algebra of finite homological dimension (in fact, we provide three of them)
such that for the associated monomial algebra, both the set of relations and
the monomial bases of the homologies are context-free languages. Namely, for each context-free language $L$ which is a homomorphic image of the Dyck language over a finite alphabet, we construct a finitely presented algebra such that its Hilbert series is calculated in terms of the generating function $H_L(z)$ of the language $L$.

The paper is organized as follows. In Section~\ref{sec:preliminaries},
we briefly recall the notions of Hilbert series of associative algebras,
monomial algebras, and context-free languages.
In Section~\ref{sec:main_results}, we discuss a construction that assigns a finitely presented algebra to such a language and give explicit formulae for the homologies and the Hilbert series of the algebra. Finally, we describe several examples of such algebras with particular choices of $L$. These are new examples of finitely presented graded algebras with non-rational algebraic Hilbert series.

\begin{acknowledgement}
	{We acknowledge the support of the University of Bari. The research
		of the author D.P. is supported by RFBR project 18-01-00908.}
\end{acknowledgement}

\section{Preliminaries}

\label{sec:preliminaries}

\subsection{Associative algebras and their Hilbert series}
\label{subs:HS}

The free monoid generated by a set $X$ is denoted by $X^*$.
Following theoretical computer science, we call the elements of $X^*$
{\em words} and the subsets of $X^*$ {\em languages}.
 
Let $A$ be a unital associative algebra over a field $k$ generated by a finite subset $X$. The words and languages then correspond to elements and subsets of $A$ which we denote by the same symbols.

Let us define a degree function on $A$ 
by putting $\deg x_i = d_i \in \Z_{>0}$ for all $x_i \in X$. Then, we put
$\deg w = d_{i_1} + \dots + d_{i_s}$ for a word $w = x_{i_1} \cdots x_{i_s}
\in X^*$ and $\deg a = \max\{ \deg w_i \}$ for an element $0 \ne a =
\sum_i c_i w_i\in A$ ($c_i\in k$).
This gives an ascending filtration  $F_0 = k 1$, $F_d =
k \{a\mid \deg a \le d \}$ on $A$. 

The Hilbert series of $A$ is then defined as the formal power series 
$$
H_A(z) = \sum_{n\ge 0} z^n \dim (F_n/F_{n-1}).
$$
The algebra $A$ is {\em graded} if it is equal to the direct sum 
$A_0\oplus A_1\oplus A_2\oplus \dots$, where $A_0 = k1$ and $A_d =
k \{w\mid \deg w = d \}$. In this case, we have $H_A(z) =
\sum_{n\ge 0} z^n \dim A_n.$

We assume now that the reader is familiar with the theory of noncommutative
\Gr\ bases which are also called \Gr-Shirshov bases (see, for instance, \cite{BC,Mo,Uf2}).
We recall here some basic foundations.
Let $I$ be a two-sided ideal of the free associative algebra
$F = k \langle X\rangle$ such that $A = F/I$. 
Suppose we have a multiplicative well-ordering $\prec$ on $X^*$. 
This gives a monomial ordering on $F$. Then, let $0\neq f =
\sum_{i=1}^s c_i w_i\in F$ with $0\neq c_i\in k, w_i\in X^*$ and
$w_1\succ w_2\succ \ldots \succ w_s$.
The word $\lm(f) = w_1$ is called the {\em leading monomial of $f$}.
A (possibly infinite) subset $U\subset I$ is called a {\em \Gr-Shirshov
basis}, briefly a {\em GS-basis of $I$}, if $\lm(U) =
\{\lm(f)\mid 0\neq f\in U\}\subset X^*$ is a monomial basis of the monomial
ideal
\[
\LM(I) = \langle \lm(f)\mid 0\neq f\in I \rangle\subset F.
\]
The GS-basis $U$ is called {\em minimal} if the monomial basis $\lm(U)$
is such. We call $\LM(I)$ the {\em leading monomial ideal of $I$}.
If $J = \LM(I)$, one defines the corresponding monomial algebra $B = F/J$.

If the algebra $A$ is graded, then it is is easy to prove
that $\HS(A) = \HS(B)$. Moreover, the same is true for non-graded algebras
if the ordering is degree-compatible, that is, $w_1\prec w_2$ provided that
$\deg w_1 < \deg w_2$. So, to compute the Hilbert series of the general
algebra $A$ it is sufficient to calculate it for the associated monomial
algebra $B$.

\subsection{The homology and Hilbert series of monomial algebras}
\label{subs:automaton}

In this sections, we recall some basic facts about monomial algebras and their homology. We have discussed this topic in details in~\cite{our_jsc_paper}. For a complete introduction, we refer the reader to~\cite{LS1,Uf2}. 

Let $A = k\la X \ra /\la L \ra$ be an associative algebra generated by a finite set $X = \{x_1 ,\dots , x_n \}$ subject to a monomial set of relations $L_1 \subset X^*$. We assume that this set of  relations is minimal, that is, the language $L_1$ is subword-free.  

The homology $\Tor_\cdot^A(k,k)$ of monomial algebras are calculated via so-called chains~\cite{An, Ba1}. More precisely, we have  
$\Tor_0^A(k,k) = k$, while for $n\ge 0$ the graded vector space $\Tor_{n+1}^A(k,k)$ is isomorphic to the span of a language $L_n$ called {\em the language of $n$-chains} of the monomial algebra $A$. Denote $L = X^* L_1 X^*$ and $X^+ = X^* \setminus \{1 \}$.
Then, for all $t\geq 1$ it holds that
\begin{equation*}
\label{govform}
\begin{array}{rcl}
L_{2t} & = & (X^+ L^t \cap L^t X^+)\setminus (X^+ L^t X^+ \cup L^{t+1}), \\
L_{2t-1} & = & (X^+ L^{t-1} X^+ \cap L^t) \setminus (X^+ L^t \cup L^t X^+).
\end{array}
\end{equation*}
In particular, $L_0 = X$, and for $t=1$ we get $L_1$ both on the left- and right-hand sides.
 
Note that the classical definition of chains is recursive and the above one is based on the Govorov's formulae for homologies of associative algebras~\cite{Go2}. We discuss these definitions in~\cite{our_jsc_paper}. 

Given a degree function on $X^*$, one can define a generating function of any language $W\subset X^*$ 
by $H_W(z) = \sum_{w\in W} z^{\deg w} $. 
Then, the Hilbert series of the graded vector space 
$\Tor_{n+1}^A(k,k)$ is equal to $H_{L_n}(z)$. 
>From the exact sequence corresponding to the minimal free resolution 
of the $A$-module $k$, 
one gets the following formula for its Hilbert series:
\begin{equation}
\label{hilbhom}
H_A(z) = \left( 1 - \sum_{i=0}^k (-1)^i H_{L_i}(z) \right)^{-1}.
\end{equation}

\subsection{Automaton algebras and homologically unambiguous algebras} 
\label{subs:unambiguous}

The automaton algebras were introduced by Ufnarovski~\cite{Uf1}. In our terms,
one can defined them as follows.
\begin{definition}
	A monomial algebra is called automaton, if the following equivalent conditions hold:
	
	(i) the set $S$ of nonzero words in $A$ is a regular language;
	
	(ii) the set of relations $L_1$ is a regular language;
	
	(iii) the chain languages $L_n$ are regular, for all $n$.
\end{definition}	
We refer the reader to~\cite{LS2, LT, Uf1, Uf2} for details on automaton algebras. In particular, the Hilbert series of each automaton algebra is a rational function. It can be calculated using the methods of formal language theory as the generating function of the regular language $S$
(see~\cite{LS2,our_jsc_paper,LT}). 

If an algebra $A$ has a non-rational Hilbert series, it cannot be automaton. So, for such algebras the condition that the corresponding languages are automaton should be weakened. A more general class of languages is the class of context-free (c-f for short) languages.
Unfortunately, there does not exist an algorithm to calculate the generating function of any c-f language. However, for some classes of c-f languages such algorithms do exist. The most important result in this direction is a theorem by Chomsky and Sch\"utzenberger stating that the generating function of each {\em unambiguous} c-f language is an algebraic function. Moreover, the theorem provides a way to construct a system of algebraic equations which defines the generating function. For a detailed description of these methods, see \cite{our_jsc_paper,KS}). 

A monomial algebra is called {\em (homologically) unambiguous} if all chain languages $L_n$ are unambiguous c-f. We refer the reader to~\cite{our_jsc_paper} for the discussion and examples of such algebras. Note that unlike the automaton case, it is not sufficient to check this condition for $L_1$
\cite[Example~5.4]{our_jsc_paper}. 
If the unambiguous algebra $A$ has finite homological dimension, then it follows from~(\ref{hilbhom}) that 
the Hilbert series $H_A(z)$ is an algebraic function.

Suppose that a finitely generated  algebra $A$ is not monomial. If the associated monomial algebra $\widehat A$ is unambiguous, then $H_A(z)$ is algebraic. In~\cite{our_jsc_paper}, we have provided an example of a finitely presented algebra of that kind such that its Hilbert series is not rational.

\section{Finitely presented algebras associated to context-free languages}

\label{sec:main_results}

\subsection{The general construction}

	Let us describe a  
	class of finitely presented associative  algebras. Each algebra of this class is related to an arbitrary context free language $L$ which is a homomorphic image of a Dyck language $D_n = D_n(a_1, b_1, \dots, a_n,b_n)$. Recall that $D_n$ consists of the words with balanced parentheses of $n$ possible kinds, where $a_i$ is the opening parenthesis and $b_i$ is the closing parenthesis of the $i$-th pair.  Note that a classical theorem by Chomsky and 
	Sch\"utzenberger provides that each context-free language is an intersection of such a language $L$ with a regular one, so that this class of languages is quite general.

	Suppose $L = \phi (D_n) \subset T^*$ with $T = \{t_1, \dots, t_m\}$. 
	Let $A$ be an algebra generated by the set $X$ of  variables 
	$a_i,b_i, a_{i}^{j}, b_{i}^{j},x,e,y, t_k$, 
	where $i,j$ run in $\{1,\ldots,n\}$ and $k$ runs in $\{1,\ldots,m\}$. 
	The relations are defined for each $i,j,l\in \{1,\ldots,n\}$ as follows:
	
	(i) $a^i_ix-xa^i_i, b^1_1 x-xe$;
	
	(ii) $a_i^j a_l-a_i a_l^j, a_i^j b_l-a_ib_l^j,  b_i^j a_l-b_ia_l^j, b_i^j b_l-b_ib_l^j, a_i^j e-a_i b_j, b_i^je-b_i b_j$;
	
	(iii) $a_i y-y \phi(a_i), b_i y-y \phi(b_i), a_i^j y, b_i^j y$;
	
	(iv) $xye$.
	
	We denote by $R$ the set of the above relations.

	Denote $d = \max\{\deg \phi(a_1), \deg \phi(b_1), \dots, \deg \phi(a_n), \deg \phi(b_n)\}$. 	
	We introduce a new degree function on the variables as follows:
	$|a_i| = |b_i| = |a_i^j| = |b_i^j| = |x| = |y| = |e| = D$ and $|t_k|=1$
	for all possible $i,j,k$.  
	Set a deglex ordering over the words on all variables by  putting $a_\bullet^\bullet >b_\bullet^\bullet>a_\bullet>b_\bullet>e>x>y> t_\bullet$.
	
	\begin{theorem}
		\label{th:main_example}
		The minimal Gr\"obner basis of the ideal generated by $R$
		is the disjoint union of the set of the relations  (i)--(iii)
		with the set of monomials
		$$
		    xP_nyLe,
		$$
		where $P_n = (D_n e)^*$ is the language consisting of the empty word and the words of the form
		$$
		   e^{i_0}w_1 e^{i_1}\dots w_s e^{i_s} 
		$$
		for all $s, i_0\ge 0$ , $i_1,\dots, i_{s} > 0$ and $w_1, \dots , w_s \in D_n$.
	\end{theorem}

\begin{proof}
	First, note that the relations (i)--(iii) form the minimal Gr\"obner basis of the ideal generated by them. Indeed, the first term of each relation is equal to its leading monomial. The only overlapping of them are between the ones of the first four types of relations  (ii) and the leading monomials ot the first two types of relations (iii). In all cases, the resulting s-polynomials are reduced to zero. For example, 
	the intersection of the leading monomials of  $a_i^j b_l-a_ib_l^j$  and
	$b_l y-y \phi(b_l)$ gives an s-polynomial
	$$
	 (a_i^j b_l-a_ib_l^j)y - a_i^j(b_l y-y \phi(b_l)) = 
	 a_i^j y \phi(b_l) -a_ib_l^j y,
	$$  
	which is immediately reduced to zero by the last two elements of (iii).
	All other cases of overlapping are similar.
	
	Now, we are ready to prove that the Gr\"obner basis mentioned in the theorem consists of the relations (i)--(iv) and a subset of $xP_n yLe$. We proceed by the induction on the length $d$ of the leading monomial of a Gr\"obner basis element which we denote by $g$. The induction base $d=2$ follows immediately.
	
	Let $d\ge 3$. The element $g$ is obtained as a complete reduction of an s-polynomial based on the elements having leading monomials with lower lengths. By the induction, the only possible (new) overlappings of such leading monomials are between the relations of type (i) and some element  $g' \in xP_n yLe$ of the Gr\"obner basis. Let $g' = x p y ev$ (with $p\in P_n, v\in L$). Then, $g$ is the complete reduction of  one of the s-polynomials
	$s_1 = a^i_i g' -  (a^i_ix-xa^i_i) p y ve = xa^i_i p y v e$
	or $s_2 = b^1_1 g' -  (b^1_1x-xe) p y ve = xe p y ve$. Here $s_2$ belongs to $xP_n yLe$, so that we can assume that $g$ is the complete reduction of 
	$s_1$.
	
	If $p=1$, then $s_1 = xa^i_i y ve$ is reduced to 0 by (iii).
	
	Suppose that $p\in D_ne$;  then either $p=e$ or $p = \tilde p b_je$ for some subword $\tilde p $ and some $j$.
	In the first case, $s_1 = xa^i_i  ey ve$ is reduced to $g = x y\phi(a_ib_i)ve \in xyL e\subset xP_nyLe$. In the second case, $s_1 = xa^i_i \tilde p b_j e y ve$ is reduced by (ii) to the monomial $xa_i \tilde p b_j b_i y ve$. This monomial is reduced by (iii) to the monomial $g = xy\phi(w)ve$, where $w = a_i \tilde p b_j b_i  \in D_n $. We see that in this case again $g \in xyLe$.
	
	Now, let $p \ne 1 $ and  $p\in P_n \setminus D_ne$. We have $p = we p'$ for some $w\in D_n, p'\in P_n e$. Then, the complete reduction $g = xa_i wb_i p'yve$ of $s_1 = xa^i_i  we p' y ve$ belongs to $xP_ne y L$.
	
	Now, it remains to prove that all elements of the set $
	xP_nyLe
	$
	belong to the minimal Gr\"obner basis. Indeed, 
	let us define a homomorphism $\alpha: \{ a_i, b_i, e \mid i=1,\ldots,n \}^* \to
	 \{ a_i^i, b_1^1, e\mid i=1,\ldots,n \}^*$ by putting 
		$\alpha: a_i \mapsto a_i^i, b_i\mapsto b_1^1, e\mapsto b_1^1$.
	Then, for each two words $v\in D_n, w\in P_n$ the element 
	$xw y ve$ is the complete reduction of the word $\alpha(wv) xy e$.
	Since this word $\alpha(wv) xye$ is divisible by the monomial (iv), it belongs to the ideal $\la R\ra$, so that the (normal) element $xw y \phi(v) e$ belongs to the minimal Gr\"obner basis,
	for all $v\in D_n, w\in P_n$.	
	\end{proof}

Note that the generating series of the languages $D_n$ and $P_n$ (with the standard degree functions) are
$$
H_{D_n}(q) = \frac{1-\sqrt{1-4nq^2}}{2nq^2}
$$
and 
$$
H_{P_n}(q) = H_{(D_ne)^*} (q) = \frac{1}{1-qH_{D_n}(q)} 
=\frac{2nq}{2nq-1+ \sqrt{1-4nq^2}}.
$$

\begin{corollary}
	$(a)$ For the associated monomial algebra $\widehat A$ to the algebra $A$ above, the graded vectors spaces $\Tor_i^{\widehat A} (k,k)$  are spanned, respectively, by the following sets $L_{i-1}$
	\begin{itemize} 
	
	\item
	$X$ (for $i=1$), 
	
	\item
	the disjoint union of the set of the first terms of (i)--(iii) with the set $xP_nyL$ (for $i=2$), 
	
	\item
	the set 
	$\{ a^1_1, a^2_2, \dots, a^n_n, b^1 \} xP_nyL$ (for $i=3$),
	
	\item
	 and the empty set  for each $i\ge 4$.
	\end{itemize}
	
	In particular, all these languages $L_{i-1}$ are c-f.
	Moreover, they are unambiguous c-f if the language $L$ is unambiguous c-f. 
	
	$(b)$ The	
	generating functions of these graded vectors spaces 
	 with respect to the degree function $|\cdot |$ are
	\begin{align*}
	   (2n^2 +2n+3) z^{d} + mz \mbox{ for } i=1, \\
	z^{2d} \left( 4n^3+4n^2 +3n+1 \right) + z^{3d} H_{P_n}(z^d) H_L(z) \mbox{ for } i=2, \\
	(n+1)z^{4d} H_{P_n}(z^d) H_L(z) \mbox{ for } i=3, 
	\\ 
	\mbox{and 0 for } i\ge 4.
\end{align*}
	
	$(c)$ Both algebras $A$ and $\widehat A$ have global dimension 3.
	
	$(d)$
	The Hilbert series of both algebras $A$ and $\widehat A$ with respect to the degree function $|\cdot |$ is 
	\begin{align*}
	\left(
	1-mz-	z^{d} (2n^2 +2n+3) +	z^{2d} ( 4n^3+4n^2 +3n+1 )
	\right.
	\\
	\left.
	+ z^{3d}
 (1-(n+1)z^d ) H_{P_n}(z^d)H_L(z) 
	\right)^{-1}.
	\end{align*}
	In particular, both algebras have exponential growth.

\end{corollary}	
	
\begin{remark}\hspace{-3pt}.\
The graded vector space $\Tor_3^{A} (k,k)$  has Hilbert series
$$
z^{3d} -  z^{3d}\left(1-(n+1)z^d\right) H_{P_n}(z^d)H_L(z).
$$
This formula follows from the equality $H_A(z)^{-1} = H_{\widehat A}(z)^{-1}$, where for each of the two 3-dimensional algebras $A$ and $\widehat A$ the inverse of the Hilbert series is equal to the Euler characteristic 
$1-H_{\Tor_1^{\cdot} (k,k)}(z)+H_{\Tor_2^{\cdot} (k,k)}(z)-H_{\Tor_3^{\cdot} (k,k)}(z)$.
\end{remark}

\subsection{Graded algebras examples}

In this subsection, we give new examples of graded finitely presented algebras with non-rational algebraic Hilbert series. These examples are based on the general construction described above.

\begin{example}\hspace{-3pt}.\
With the above notations, let $L $ be the Dyck langauge on the alphabet $T = \{t_1, \dots , t_{2n} \}$, so that $\phi$ is the obvious isomorphism and $d=1$. Then, the algebra $A$ is graded with the standard degree function with Hilbert series
	\begin{align*}
\left(
1-	z (2n^2 +4n+3) +	z^{2} ( 4n^3+4n^2 +3n+1 )
\right.
\\
\left.
+ z^{3}
(1-(n+1)z ) H_{P_n}(z)H_{D_n}(z) 
\right)^{-1}.
\end{align*}
Note that here
$$
H_{P_n}(z)H_{D_n}(z) = \frac{1-\sqrt{1-4nz^2}}{z\left(2nz-1+\sqrt{1-4nz^2}\right) } = {\frac {1-2z-\sqrt {-4\,n{z}^{2}+1}}{2z \left( nz+z-1 \right) }
}
,
$$	
so that the Hilbert series
	\begin{align*}
 H_A(z) =\left(1-	z (2n^2 +4n+3) +	z^{2} ( 4n^3+4n^2 +3n+\frac{1}{2} ) +z^3
 \right.
 \\
 \left.
 +
  \frac{z^2\sqrt{1-4\,n{z}^{2}}}{2} 
  \right)^{-1}
\end{align*}
 is not rational. This algebra is homologically unambiguous.
\end{example}

\begin{example}\hspace{-3pt}.\
	Now, let $L $ be the language consisting of all words on $t_1$ and $t_2$ with the same number of $t_1$-s and $t_2$-s. It is the image of $D_2$ under the homomorphism $\phi: a_1\mapsto t_1,
	b_1\mapsto t_2, a_2\mapsto t_2,
	b_2\mapsto t_1$. So, here $d=1$ (so that the algebra $A$ is graded with the standard grading) and $n=m=2$. Then, the Hilbert series of $A$ is 
		\begin{align*}
	\left(
	1-	17 z +	55 z^{2}
	+ z^{3}
	(1-3z ) H_{P_2}(z)H_{L}(z) 
	\right)^{-1},
	\end{align*}
	where
	$
	H_{L}(z) = \sum_{n\ge 0} \binom{2n}{n}z^{2n} = 
	{1}/{\sqrt{1-4z^2}}.
	$	
\end{example}

\begin{example}\hspace{-3pt}.\
	Finally, let $L$ be the language over the 26 capital letter alphabet 
	consisting of all words with balanced pairs of the words "BEG", "END"
	and of the words "FOR, END". It is the image of $D_2$ under the homomorphism $\phi: a_1\mapsto BEG,
	b_1\mapsto END, a_2\mapsto FOR,
	b_2\mapsto END$. Then, we have $n=2, m=26, d=3$. Since all generators of $L$ have the same length $3$, the algebra $A$ is graded with the Hilbert series 
	\begin{align*}
\left(
1-26z- 15	z^{3}  + 55	z^{6} 
+ z^{9}
(1-3 z^3 ) H_{P_2}(z^3)H_L(z) 
\right)^{-1},
\end{align*}	 
where
$$
H_L(z) = D_2(z^3) = \frac{1-\sqrt{1-8z^6}}{4z^6}.
$$	
In the explicit form, we have therefore 
$$
H_A(z) = \left(
1-26z- 15	z^{3}  + \frac{109}{2}	z^{6}+z^9+z^{6}
 \frac {\sqrt {1-8 {z}^{6}}}{ 2 }
\right)^{-1}.
$$
\end{example}


\begin{thebibliography}{99}
	
\bibitem{An} Anick, D., {\em On the homology of associative algebras},
Trans. Amer. Math. Soc., {\bf 296} (1986), 2, 641--659.

\bibitem{Ba1} Backelin, J., {\em La s\'erie de Poincar\'e-Betti d'une
alg\`ebre gradu\'ee de type fini \`a une relation est rationnelle},
C. R. Acad. Sei. Paris. S\'er. A, {\bf 287} (1978), 843--846.	
	
\bibitem{Bell_prime_PI} Bell, Jason P., {\em The Hilbert series of prime PI rings}, Israel Journal of Mathematics, {\bf 139} (2004), 1--10.
	
\bibitem{Belov_rat_PI}  Belov, A. Y., {\em On the rationality of Hilbert series of relatively free algebras}, Russian Mathematical Surveys, {\bf 52} (1997), 394.

\bibitem{BC} Bokut, L.A.; Chen, Yuqun, {\em \Gr-Shirshov bases and their
calculation}, Bull. Math. Sci., {\bf 4} (2014), 325--395. 
	
\bibitem{Go1} Govorov, V.E., {\em Graded algebras} (Russian), Mat. Zametki,
{\bf 12} (1972), 197--204; translation in Math. Notes, {\bf 12} (1972), 552--556.

\bibitem{Go2} Govorov, V.E., {\em The global dimension of algebras} (Russian),
Mat. Zametki {\bf 14} (1973), 399--406; translation in Math. Notes {\bf 14} (1973), 789--792.

\bibitem{IS} Iyudu, N.; Shkarin, S. {\em Automaton algebras and intermediate growth},
J. of Combinatorial algebra, {\bf 2} (2018), 147--167.

\bibitem{LS1} La Scala, R., {\em Computing minimal free resolutions of right modules
over noncommutative algebras}, J. Algebra 478 (2017), 458--483.

\bibitem{LS2} La Scala, R., {\em Monomial right ideals and the Hilbert series of noncommutative
modules}, J. Symb. Comp.,  {\bf 80}, (2017), 403--415.

\bibitem{our_jsc_paper} La Scala, R.; Piontkovski, D.; Tiwari, S. K. 
{\em Noncommutative algebras, context-free grammars and algebraic Hilbert series}, J. Symb. Comp.,
Available online; \mbox{https://doi.org/10.1016/j.jsc.2019.06.005}

\bibitem{LT} La Scala, R.; Tiwari, S. K., Multigraded Hilbert series of noncommutative modules,
J. Algebra, {\bf 516} (2018), 514--544.

\bibitem{Ko} Kobayashi, Y., {\em Another graded algebra with a nonrational Hilbert
	series}, Proc. Amer. Math. Soc., {\bf 81} (1981), 19--22.

\bibitem{Kc1} Ko\c{c}ak, D., {\em Finitely presented quadratic algebras of intermediate growth},
Algebra Discrete Math., {\bf 20} (2015), 69--88.

\bibitem{Kc2} Ko\c{c}ak, D., {\em Intermediate growth in finitely presented algebras},
Internat. J. Algebra Comput., {\bf 27} (2017), 391--401.

\bibitem{KL} Krause, G.R.; Lenegan, T.H., {\em Growth of algebras and Gelfand-Kirillov dimension.
	Revised edition}, Graduate Studies in Mathematics, 22. American Mathematical Society,
Providence, RI, 2000.

\bibitem{KS} Kuich, W.; Salomaa, A.,  {\em Semirings, automata, languages},
EATCS Monographs on Theoretical Computer Science, 5. Springer-Verlag, Berlin, 1986.

\bibitem{Mo} Mora, T., {\em An introduction to commutative and noncommutative Gr\"obner
	bases}, Second International Colloquium on Words, Languages and Combinatorics (Kyoto, 1992).
Theoret. Comput. Sci., {\bf 134} (1994), 131--173. 

\bibitem{P2019} Piontkovski, D., {\em Algebras of linear growth and the dynamical Mordell-Lang conjecture}, Advances in Mathematics, {\bf 343} (2019), 141--156.

\bibitem{Sh} Shearer, J., A graded algebra with a nonrational Hilbert series,
J. Algebra, {\bf 62} (1980), 228--231.

\bibitem{Uf1} Ufnarovski, V.A., {\em On the use of graphs for calculating the basis,
	growth and Hilbert series of associative algebras} (Russian), Mat. Sb., 180 (1989),
1548--1560; translation in Math. USSR-Sb., {\bf 68}, (1991), 417--428.

\bibitem{Uf2} Ufnarovski, V.A., {\em Combinatorial and asymptotic methods in algebra},
Algebra VI, 1--196, Encyclopaedia Math. Sci., 57, Springer, Berlin, 1995.


	
\end{thebibliography}
\end{document}